\documentclass[11pt]{amsart}

\usepackage{amscd,amsmath,latexsym,amsthm,amsfonts,amssymb,graphicx,geometry}

\usepackage{aliascnt}

\usepackage{enumerate}

\usepackage[dvipsnames]{xcolor}

\usepackage{hyperref}
\hypersetup{
	colorlinks=true,
	linkcolor=blue,
	citecolor=cyan,
	urlcolor=NavyBlue,
}

\allowdisplaybreaks

\usepackage[norefs,nocites]{refcheck}

\usepackage{geometry}
\providecommand{\U}[1]{\protect\rule{.1in}{.1in}}

\definecolor{darkgreen}{rgb}{0.0, 0.5, 0.0}
\providecommand{\U}[1]{\protect\rule{.1in}{.1in}}
\providecommand{\U}[1]{\protect\rule{.1in}{.1in}}
\providecommand{\U}[1]{\protect\rule{.1in}{.1in}}
\providecommand{\U}[1]{\protect\rule{.1in}{.1in}}
\providecommand{\U}[1]{\protect\rule{.1in}{.1in}}
\providecommand{\U}[1]{\protect\rule{.1in}{.1in}}

\newtheorem{theorem}{Theorem}[section]
\newaliascnt{acknowledgement}{theorem}

\aliascntresetthe{acknowledgement}

\newaliascnt{algorithm}{theorem}

\aliascntresetthe{algorithm}

\newaliascnt{axiom}{theorem}

\aliascntresetthe{axiom}

\newaliascnt{case}{theorem}

\aliascntresetthe{case}

\newaliascnt{claim}{theorem}

\aliascntresetthe{claim}

\newaliascnt{conclusion}{theorem}

\aliascntresetthe{conclusion}

\newaliascnt{condition}{theorem}

\aliascntresetthe{condition}

\newaliascnt{conjecture}{theorem}

\aliascntresetthe{conjecture}

\newaliascnt{corollary}{theorem}

\aliascntresetthe{corollary}

\newaliascnt{criterion}{theorem}

\aliascntresetthe{criterion}

\newaliascnt{definition}{theorem}
\newtheorem{definition}[definition]{Definition}
\aliascntresetthe{definition}

\newaliascnt{lemma}{theorem}
\newtheorem{lemma}[lemma]{Lemma}
\aliascntresetthe{lemma}

\newaliascnt{notation}{theorem}

\aliascntresetthe{notation}

\newaliascnt{proposition}{theorem}

\aliascntresetthe{proposition}

\newaliascnt{solution}{theorem}

\aliascntresetthe{solution}

\newaliascnt{summary}{theorem}

\aliascntresetthe{summary}

\theoremstyle{definition}
\newaliascnt{example}{theorem}

\aliascntresetthe{example}

\newaliascnt{exercise}{theorem}

\newaliascnt{problem}{theorem}

\aliascntresetthe{problem}

\newaliascnt{remark}{theorem}
\newtheorem{remark}[remark]{Remark}
\aliascntresetthe{remark}

\begin{document}

		\title[Operator ranges and spaceability]{Operator Ranges and Spaceability: Extending a Result of Kitson and Timoney}

		\author[G. Ribeiro]{Geivison Ribeiro}
		\address[G.~Ribeiro]{Departamento de Matem\'atica \newline\indent
			Universidade Estadual de Campinas \newline\indent
			13083-970 -- Campinas, Brazil.}
		\email{\href{mailto:geivison@unicamp.br}{geivison@unicamp.br} \textrm{and} \href{mailto:geivison.ribeiro@academico.ufpb.br}{geivison.ribeiro@academico.ufpb.br}}
		
		
	\begin{abstract}
		We revisit the results of Kitson and Timoney \emph{[J.~Math.~Anal.~Appl.~\textbf{378} (2011), 680--686]} on the spaceability of complements of operator ranges, extending one of their main theorems to the general Fr\'echet setting. In particular, we provide an affirmative answer to the question posed in \emph{Remark~3.4} of that paper, showing that the conclusion remains valid when the operators act between Fr\'echet spaces. Moreover, we show that the same phenomenon occurs for arbitrary (possibly uncountable) families of operators. The arguments presented here follow the spirit of the original work.
	\end{abstract}

	\maketitle
	
\section{Introduction}

The notions of \emph{lineability} and \emph{spaceability} trace back to the pioneering ideas of Gurariy~\cite{Gurariy} and were later formalized, in their modern form, by Aron, Gurariy and Seoane-Sepúlveda in 2005~\cite{AGS}. Since then, these topics have become an active field of research in Analysis (and particularly in Functional Analysis), with a systematic treatment given in the 2016 monograph by Aron, Bernal-González, Pellegrino and Seoane-Sepúlveda~\cite{AGPS}. Over the past decades, several general criteria and refinements have been established; see, among others, Bernal-González and Ordóñez-Cabrera (2014)~\cite{bernaljfa}, Fávaro, Pellegrino and Tomaz (2020)~\cite{FPT}, and Calderón-Moreno, Gerlach-Mena and Prado-Bassas (2023)~\cite{MMB}, as well as more recent contributions in~\cite{Araujo,ABRR,ABRR-2025,PR,Raposo,DR,RSR,Leo}.

A conceptual milestone in this theory is due to Drewnowski (1984)~\cite{Drewnowski1984}: if $T:E\to F$ is a continuous linear operator between Fréchet spaces with non-closed range, then the complement $(F\setminus T(E))\cup\{0\}$ contains a closed infinite-dimensional subspace. This result has become a central tool in the study of large linear structures within nonlinear sets.

Following this perspective, Kitson and Timoney (2011)~\cite{KitsonTimoney2011} proved that, for a Fréchet space $X$ and a countable family of continuous linear operators $(T_n:Z_n\to X)_{n\in\mathbb{N}}$ (with each $Z_n$ being a Banach space), the complement of the linear span of the ranges is \emph{spaceable} whenever that span is not closed. The authors also asked whether their result extends to the case in which the domain spaces $Z_n$ are themselves Fréchet (see \emph{Remark~3.4} in~\cite{KitsonTimoney2011}).

The present paper gives an affirmative answer to this question. We show that the same result remains valid for families of operators $T_n:Z_n\to X$ acting between Fréchet spaces, by encoding the family into the range of a single operator and applying geometric arguments in the spirit of~\cite{Drewnowski1984}. Moreover, we remove the countability assumption and show that the same conclusion holds for \emph{arbitrary} (possibly uncountable) families of operators.

\section{Analytic and geometric preliminaries}

In this section we collect auxiliary notions and tools that will be used throughout the paper.

\begin{definition}{\cite[Definition~8.1.6]{BotelhoPellegrinoTeixeira2025}}
	Let $E$ be a topological vector space over $\mathbb{K}\in\{\mathbb{R},\mathbb{C}\}$. 
	A set $B\subset E$ is said to be \emph{balanced} if $\lambda x\in B$ for every $x\in B$ and every $\lambda\in\mathbb{K}$ with $|\lambda|\le1$.
\end{definition}

\medskip
The next two lemmas are elementary analytical tools ensuring lower and monotone semicontinuity for infinite series of nonnegative terms. They are included for completeness since they will be repeatedly used in the proof of the main theorem.

\begin{lemma}\label{lem:fatou-series}
	Let $(a_n^{(t)})_{n,t\in\mathbb{N}}$ be a family of nonnegative numbers, and define
	$a_n:=\liminf_{t\to\infty}a_n^{(t)}$ (possibly $+\infty$). Then
	\[
	\sum_{n=1}^\infty a_n \;\le\; \liminf_{t\to\infty}\sum_{n=1}^\infty a_n^{(t)}.
	\]
\end{lemma}

\begin{lemma}\label{lem:monotone-series}
	Let $(b_{r,n})_{r,n\in\mathbb{N}}$ be a family of nonnegative numbers such that
	$b_{r,n}\uparrow b_n$ as $r\to\infty$ for each $n$. Then
	\[
	\sum_{n=1}^\infty b_n \;=\; \lim_{r\to\infty}\sum_{n=1}^\infty b_{r,n}.
	\]
\end{lemma}

The proofs follow directly from the classical Fatou and monotone convergence theorems for nonnegative series and are therefore omitted.

\medskip
We now recall a fundamental geometric result underlying our approach. 

\begin{theorem}[{\cite[Prop.~2.4]{KitsonTimoney2011}; see also \cite{Drewnowski1984}}]\label{thm:drewnowski}
	Let $E$ and $F$ be Fr\'echet spaces, and let $T:E\to F$ be a continuous linear operator.
	If $T(E)$ is not closed in $F$, then the complement $F\setminus T(E)$
	contains a closed infinite-dimensional subspace.
\end{theorem}

In fact, this geometric statement was originally proved by Drewnowski~\cite{Drewnowski1984}, 
under the assumption that $T(E)$ is not relatively open or not barrelled. 
Here we follow the presentation adopted by Kitson and Timoney~\cite{KitsonTimoney2011}, 
who used this result as a central tool in their study of spaceable complements of operator ranges.

\section{The main result}

We now state and prove the Fr\'echet version of Theorem~3.3 in \cite{KitsonTimoney2011}.

\begin{theorem}\label{thm:Frechet-KT}
	Let $(Z_n)_{n\in\mathbb{N}}$ be Fr\'echet spaces. 
	If $X$ is a Fr\'echet space and $T_n:Z_n\to X$ is a family of continuous linear operators, and
	\[
	Y:=\operatorname{span}\Big(\bigcup_{n\in\mathbb{N}}T_n(Z_n)\Big)
	\]
	is not closed in $X$, then the complement $X\setminus Y$ is spaceable.
\end{theorem}

\begin{proof}
	We begin by fixing countable families of seminorms $(\|\cdot\|_k)_{k=1}^{\infty}$ on $X$ and, for each $n\in\mathbb{N}$, families $(\|\cdot\|_{n,m})_{m=1}^{\infty}$ on $Z_n$ generating their respective topologies. 
	Consider the countable product $\prod_{n=1}^{\infty}Z_n$ endowed with the product topology, and define
	\[
	E:=\Big\{z=(z_n)_{n=1}^{\infty}\in\textstyle\prod_{n=1}^{\infty}Z_n:\ 
	\sum_{n=1}^{\infty}\|T_nz_n\|_k<\infty\ \forall k,\ 
	\sup_{r\in\mathbb{N}}\sum_{n=1}^{\infty}2^{-n}\!\!\sum_{m=1}^r2^{-m}\|z_n\|_{n,m}<\infty\Big\}.
	\]
	
	We first verify that $E$ is infinite-dimensional.  
	For each $n\in\mathbb{N}$, fix $u_n\in Z_n\setminus\{0\}$ and define $e^{(n)}=(0,\ldots,0,u_n,0,\ldots)$.  
	Then $\tau_k(e^{(n)})=\|T_nu_n\|_k<\infty$ and $\rho_r(e^{(n)})=2^{-n}\sum_{m=1}^r2^{-m}\|u_n\|_{n,m}<\infty$, hence $e^{(n)}\in E$.  
	Since these vectors have disjoint supports, they are linearly independent, and therefore $\dim E=\infty$.
	
	Next we show that $E$ is a Fr\'echet space.  
	For $k,r\in\mathbb{N}$ and $z\in E$, define
	\[
	\tau_k(z):=\sum_{n=1}^{\infty}\|T_nz_n\|_k,
	\qquad
	\rho_r(z):=\sum_{n=1}^{\infty}2^{-n}\sum_{m=1}^r2^{-m}\|z_n\|_{n,m}.
	\]
	The countable family $\mathcal P:=\{\tau_k:k\in\mathbb{N}\}\cup\{\rho_r:r\in\mathbb{N}\}$ defines a locally convex, Hausdorff and metrizable topology on $E$ (see \cite[Theorem~8.3.3]{BotelhoPellegrinoTeixeira2025}).  
	To verify completeness, let $(z^{(s)})_{s=1}^{\infty}\subset E$ be Cauchy for all seminorms in $\mathcal P$.  
	For each $r\in\mathbb{N}$ we have
	\[
	\rho_r(z^{(s)}-z^{(t)})=\sum_{n=1}^{\infty}2^{-n}\sum_{m=1}^r2^{-m}\|z^{(s)}_n-z^{(t)}_n\|_{n,m}\xrightarrow[s,t\to\infty]{}0.
	\]
	Fix $s\in\mathbb{N}$ and, for each $n\in\mathbb{N}$, choose $z_n\in Z_n$ such that $z^{(t)}_n\to z_n$ in $Z_n$; set $z:=(z_n)_{n=1}^{\infty}$.  
	By applying \autoref{lem:fatou-series} and \autoref{lem:monotone-series}, we obtain $\rho_r(z^{(s)}-z)\to0$ for all $r$, and hence $\rho(z^{(s)}-z)\to0$, where $\rho:=\sum_{n=1}^{\infty}2^{-n}\sum_{m=1}^{\infty}2^{-m}\|\cdot\|_{n,m}$.  
	Similarly, for each $k\in\mathbb{N}$,
	\[
	\tau_k(z^{(s)}-z)
	=\sum_{n=1}^{\infty}\|T_n(z^{(s)}_n-z_n)\|_k
	\le\liminf_{t\to\infty}\tau_k(z^{(s)}-z^{(t)})\xrightarrow[s\to\infty]{}0.
	\]
	Thus $z\in E$ and $z^{(s)}\to z$ for all seminorms in $\mathcal P$, proving that $E$ is Fr\'echet.
	
	We now define the operator
	\[
	S:E\to X,\qquad S(z):=\sum_{n=1}^{\infty}T_nz_n.
	\]
	The series converges in $X$ since $\sum_{n=1}^{\infty}\|T_nz_n\|_k<\infty$ for each $k$, and $\|S(z)\|_k\le\tau_k(z)$, which shows that $S$ is well defined and continuous.
	
	We next describe the range of $S$.  
	Let
	\[
	D:=\{z=(z_n)_{n=1}^{\infty}\in E:\ \exists\,N\in\mathbb{N}\text{ such that }z_n=0\text{ for all }n>N\}.
	\]
	If $z\in D$ has support contained in $\{1,\dots,N\}$, then $S(z)=\sum_{n=1}^N T_nz_n\in Y$, so $S(D)\subset Y$.  
	Conversely, given $y=\sum_{i=1}^mT_{n_i}w_i$ (with distinct $n_i$ and $w_i\in Z_{n_i}$), define $z\in D$ by $z_{n_i}=w_i$ and $z_n=0$ otherwise; then $S(z)=y$.  
	Hence $S(D)=Y$.
	
	We now turn to the spaceability argument.  
	If $S(E)$ is not closed in $X$, then \autoref{thm:drewnowski} yields a closed infinite-dimensional subspace $M\subset X$ such that $M\setminus\{0\}\subset X\setminus S(E)$.  
	Since $Y=S(D)\subset S(E)$, it follows that $X\setminus Y\supset M\setminus\{0\}$, and therefore $X\setminus Y$ is spaceable.
	
	Assume now that $S(E)$ is closed in $X$.  
	Choose $x_\ast\in\overline{Y}\setminus Y$ and a sequence $(y_j)_{j=1}^{\infty}\subset Y$ with $y_j\to x_\ast$.  
	Let $\widetilde X:=\overline{\operatorname{span}}\{x_\ast,y_1,y_2,\dots\}$, a closed separable subspace of $X$, and set $\widetilde Y:=Y\cap\widetilde X$.  
	Then $x_\ast\in\overline{\widetilde Y}^{\,\widetilde X}\setminus\widetilde Y$, so $\widetilde Y$ is not closed in $\widetilde X$.  
	Working inside $\widetilde X$, any closed subspace $F\subset\widetilde X$ satisfying $F\setminus\{0\}\subset\widetilde X\setminus\widetilde Y$ is also closed in $X$ and verifies $F\setminus\{0\}\subset X\setminus Y$.
	
	For each $N\in\mathbb{N}$, define $S^{(N)}:E\to X$ by $S^{(N)}(z)=\sum_{n=1}^N T_nz_n$, let $W_N:=(S^{(N)})^{-1}(\widetilde X)$, and write $\widetilde S^{(N)}:=S^{(N)}|_{W_N}:W_N\to\widetilde X$.  
	Then
	\begin{equation}\label{eq:Ytilde-as-union}
		\widetilde Y=\bigcup_{N=1}^{\infty}\widetilde S^{(N)}(W_N).
	\end{equation}
	For each $y\in\widetilde S^{(N)}(W_N)\setminus\{0\}$, take a convex balanced neighborhood $W_y$ of $0$ in $\widetilde X$ such that $-y\notin\overline{W_y}^{\,\widetilde X}$, and define
	\[
	O_{N,y}:=(y+W_y)\cap\widetilde S^{(N)}(W_N).
	\]
	Then $O_{N,y}$ is convex, open in the subspace topology, and satisfies $0\notin\overline{O_{N,y}}^{\,\widetilde X}$.  
	Since $\widetilde S^{(N)}(W_N)$ is separable, it is Lindel\"of; take a countable subcover $\{O_{N,j}\}_{j=1}^{\infty}$ and set $C_{N,j}:=\overline{O_{N,j}}^{\,\widetilde X}$.  
	Each $C_{N,j}$ is closed, convex and avoids the origin, and
	\[
	\widetilde S^{(N)}(W_N)\setminus\{0\}
	=\bigcup_{j=1}^{\infty}\bigl(C_{N,j}\cap\widetilde S^{(N)}(W_N)\bigr).
	\]
	By \eqref{eq:Ytilde-as-union},
	\[
	\widetilde Y\setminus\{0\}
	=\bigcup_{N=1}^{\infty}\bigcup_{j=1}^{\infty}\bigl(C_{N,j}\cap\widetilde S^{(N)}(W_N)\bigr),
	\]
	a countable union of closed convex subsets of $\widetilde X$ avoiding the origin.  
	Since $\widetilde Y$ is not closed in $\widetilde X$, it has infinite codimension in $\widetilde X$, so there exists an infinite-dimensional subspace $W\subset\widetilde X$ with $W\cap\widetilde Y=\{0\}$.  
	By \cite[Corollary~5.5]{Drewnowski1984}, there exists a closed infinite-dimensional subspace $F\subset\widetilde X$ with $\dim(F\cap W)=\infty$ and $F\cap\widetilde Y=\{0\}$.  
	As $\widetilde X$ is closed in $X$, the same $F$ is closed in $X$, and
	\[
	F\setminus\{0\}\subset\widetilde X\setminus\widetilde Y\subset X\setminus Y.
	\]
	Therefore $X\setminus Y$ is spaceable.
\end{proof}

\section{Extension to the uncountable case}\label{sec:uncountable}

The statement of \autoref{thm:Frechet-KT} was given for countable families
$(T_n:Z_n\to X)_{n=1}^{\infty}$. We now show that the same conclusion remains valid for \emph{uncountable}
families of continuous linear operators.

\begin{remark}[General form of \autoref{thm:Frechet-KT}]
	Let $X$ be a Fr\'echet space and let $\{Z_i\}_{i\in I}$ be an arbitrary (possibly uncountable)
	family of Fr\'echet spaces.
	For each $i\in I$, let $T_i:Z_i\to X$ be continuous and linear, and define
	\[
	Y:=\operatorname{span}\Bigl(\bigcup_{i\in I}T_i(Z_i)\Bigr).
	\]
	If $Y$ is not closed in $X$, then the complement $X\setminus Y$ is spaceable.
\end{remark}

\begin{proof}
	Since $X$ is Fr\'echet, its topology is metrizable by a complete metric. In particular, if $Y$ is not closed in $X$, then there exist $x_\ast\in \overline{Y}\setminus Y$ and a sequence $(y_j)_{j=1}^{\infty}\subset Y$ such that $y_j\to x_\ast$ in $X$. Each $y_j$ is a finite linear combination of vectors of the form $T_i z_i$, that is,
	\[
	y_j=\sum_{k=1}^{m_j} T_{i_{j,k}}(z_{j,k}),
	\qquad\text{with }\, m_j\in\mathbb{N},\ i_{j,k}\in I,\ z_{j,k}\in Z_{i_{j,k}}.
	\]
	Set
	\[
	I_0:=\bigcup_{j=1}^{\infty}\bigl\{\, i_{j,k} : 1\le k\le m_j \,\bigr\}.
	\]
	Then $I_0$ is countable. Define
	\[
	Y_0:=\operatorname{span}\Bigl(\bigcup_{i\in I_0} T_i(Z_i)\Bigr)\subset Y.
	\]
	By construction, $(y_j)_{j=1}^{\infty}\subset Y_0$ and $y_j\to x_\ast$ in $X$, hence
	$x_\ast\in \overline{Y_0}^{\,X}\setminus Y_0$ and therefore $Y_0$ is not closed in $X$.
	
	We now restrict the discussion to a separable closed subspace of $X$ that already contains the obstruction to closedness. Let
	\[
	\widetilde{X}:=\overline{\operatorname{span}}\{x_\ast, y_1,y_2,\dots\}\subset X.
	\]
	Then $\widetilde{X}$ is a closed separable subspace of $X$, and
	\[
	\widetilde{Y}_0:=Y_0\cap \widetilde{X}
	\]
	is not closed in $\widetilde{X}$, since $y_j\in \widetilde{Y}_0$ for all $j\in\mathbb{N}$ and $y_j\to x_\ast\in \widetilde{X}$ with $x_\ast\notin \widetilde{Y}_0$.
	
	For each $i\in I_0$, consider the closed subspace $Z_i':=T_i^{-1}(\widetilde{X})\subset Z_i$ and the restricted operator $T_i':=T_i|_{Z_i'}:Z_i'\to \widetilde{X}$. Since $\widetilde{X}$ is closed and $T_i$ is continuous, each $Z_i'$ is a closed subspace of $Z_i$; in particular, $Z_i'$ is Fr\'echet. Moreover, each $T_i'$ is continuous and linear. By construction,
	\[
	\operatorname{span}\Bigl(\bigcup_{i\in I_0} T_i'(Z_i')\Bigr)
	=\Bigl(\operatorname{span}\Bigl(\bigcup_{i\in I_0} T_i(Z_i)\Bigr)\Bigr)\cap \widetilde{X}
	=Y_0\cap \widetilde{X}
	=\widetilde{Y}_0.
	\]
	As noted, $\widetilde{Y}_0$ is not closed in $\widetilde{X}$. We can therefore apply \autoref{thm:Frechet-KT} \emph{inside $\widetilde{X}$} to the countable family $(T_i':Z_i'\to \widetilde{X})_{i\in I_0}$ and conclude that there exists a closed infinite-dimensional subspace $F\subset \widetilde{X}$ such that
	\[
	F\setminus\{0\}\subset \widetilde{X}\setminus \widetilde{Y}_0.
	\]
	Since $\widetilde{X}$ is closed in $X$, the same $F$ is closed in $X$. Finally, because $\widetilde{Y}_0=Y_0\cap \widetilde{X}\subset Y\cap \widetilde{X}$, we have
	\[
	\widetilde{X}\setminus \widetilde{Y}_0
	\supset \widetilde{X}\setminus (Y\cap \widetilde{X})
	=\widetilde{X}\cap (X\setminus Y),
	\]
	and thus
	\[
	F\setminus\{0\}\subset \widetilde{X}\setminus \widetilde{Y}_0 \subset X\setminus Y.
	\]
	Therefore $X\setminus Y$ contains the closed infinite-dimensional subspace $F\setminus\{0\}$ and is spaceable, as claimed.
\end{proof}

\vskip 5mm

\noindent\textbf{Acknowledgments.} The author was supported by CAPES --- Coordena\c{c}\~ao de Aperfei\c{c}oamento de Pessoal de N\'ivel Superior (Brazil) --- through a postdoctoral fellowship at IMECC, Universidade Estadual de Campinas (PIPD/CAPES; Finance Code 001).

\end{document}